\documentclass{article}
\usepackage[utf8]{inputenc}
\usepackage{enumerate}
\usepackage{amssymb, bm}
\usepackage{amsthm}
\usepackage{amsmath}
\usepackage{tikz-cd} 
\usepackage{mathtools}
\usepackage{mathrsfs}  
 \usepackage[all]{xy}
\title{The Bogomolov inequality on a singular complex space}
\author{Xiaojun WU}
\date{\today}
\newtheorem{mythm}{Theorem}
\newtheorem{mylem}{Lemma}
\newtheorem{myprop}{Proposition}
\newtheorem{mycor}{Corollary}
\newtheorem{mydef}{Definition}
\newtheorem{myrem}{Remark}

\begin{document}
\def \coh{\mathrm{coh}}
\def\cI{\mathcal{I}}
\def\Z{\mathbb{Z}}
\def\Q{\mathbb{Q}}  \def\C{\mathbb{C}}
 \def\R{\mathbb{R}}
 \def\N{\mathbb{N}}
 \def\H{\mathbb{H}}
  \def\P{\mathbb{P}}
 \def\cS{\mathcal{S}}
  \def\cC{\mathcal{C}}
  \def\cF{\mathcal{F}}
  \def\Hom{\mathrm{Hom}}
  \def\d{\partial}
 \def\dbar{{\overline{\partial}}}
\def\dzbar{{\overline{dz}}}
 \def\ii{\mathrm{i}}
  \def\d{\partial}
 \def\dbar{{\overline{\partial}}}
\def\dzbar{{\overline{dz}}}
\def \ddbar {\partial \overline{\partial}}
\def\cD{\mathcal{D}}
\def\cE{\mathcal{E}}  \def\cO{\mathcal{O}}
\def\P{\mathbb{P}}
\def\cI{\mathcal{I}}
\def \sim{\mathrm{sim}}
\def \reg{\mathrm{reg}}
\def \sing{\mathrm{sing}}
\def \an{\mathrm{an}}
\def \dR{\mathrm{dR}}
\def \top{\mathrm{top}}
\def \dim{\mathrm{dim}}
\def \Todd{\mathrm{Todd}}
\def \tors{\mathrm{Tors}}
\def \det{\mathrm{det}}
\def \rk{\mathrm{rk}}
\def \ch{\mathrm{ch}}
\def \cH{\mathscr{H}}
\def \ker{\mathrm{Ker}}
\def \rank{\mathrm{rank}}
\def \nn{\mathrm{nn}}
\def \im{\mathrm{Im}}
\def\cG{\mathcal{G}}
\def\cQ{\mathcal{Q}}
\maketitle
\begin{abstract}
In this note, we prove the Bogomolov inequality over a reduced, compact, irreducible, K\"ahler complex space that is smooth in codimension $2$.
The proof is obtained by a reduction to the smooth case, using Hironaka's resolution of singularities.
\end{abstract}
\section{Introduction}
In \cite{Bog78}, Bogomolov proved the celebrated inequality asserting that the discriminant
$$\Delta(E):=2rc_2(E)-(r-1)c^2_1(E)$$
of any rank $r$ vector bundle $E$ on a projective surface that is slope $H$-polystable with respect to an ample divisor $H$ is non-negative.
One analytic method of proof consists of using the existence of a Hermitian-Einstein metric on $E$.
This was proven by Simon Donaldson for projective algebraic surfaces \cite{Don85} and later for projective algebraic manifolds \cite{Don87}, by Karen Uhlenbeck and Shing-Tung Yau \cite{UY} for compact Kähler manifolds.

In the fundamental paper of Bando-Siu \cite{BS}, the following Bogomolov inequality for a coherent sheaf is proven.
\begin{mythm}
For a polystable reflexive sheaf $\cF$ over a compact K\"ahler manifold $(X, \omega)$, we have the Bogomolov inequality
$$\int_X (2r c_2(\cF)-(r-1)c_1(\cF)) \wedge \omega^{n-2} \geq 0$$
where $r$ is the generic rank of $\cF$.
Moreover, the equality holds if and only if $\cF$ is locally free and its Hermitian-Einstein metric (which exists by the work of Uhlenbeck-Yau \cite{UY}) gives a projectively flat connection.
\end{mythm}
A natural question is whether the same result holds over a singular compact K\"ahler space.

A first difficulty to be overcome in the singular case is to adopt a suitable cohomology theory to work with.
It is well-known that in general, the Poincaré lemma of the De Rham cohomology theory does not hold on a singular space.
In particular, the complex of smooth forms does not necessarily provide a resolution of the locally constant sheaf, and, as a consquence, the de Rham cohomology may be different from singular cohomology.

A second difficulty is to appropriately define Chern classes for a coherent sheaf over a singular complex space.
By an example of Voisin \cite{Voi},  over a smooth compact complex space,
a coherent sheaf does not necessarily admit a resolution by holomorphic vector bundles.
In particular, we cannot reduce the definition of Chern classes of coherent sheaves to the case of holomorphic vector bundles.
A satisfactory theory of the Riemann-Roch-Grothendieck formula has been developed over complex spaces in terms of homology rather than cohomology groups.
For lower degrees and over mildly singular spaces, although the Poincaré duality does not necessarily hold, one can show that the homology Chern class of a coherent sheaf defined by the Riemann-Roch-Grothendieck formula over complex spaces is equal to the cap product of Chern classes in singular cohomology with the fundamental homology class of the complex space.
Philosophically, the main reason is that the difference is supported in an analytic set of higher codimension. Thus for lower degrees, the difference has to be trivial.

Once Chern classes have been defined adequately, we take a resolution of singularities to reduce the proof of the Bogomolov inequality over a complex space smooth in codimension 2 to that of a smooth model. More precisely, we have obtained the following result.
Let $\pi: \tilde{X} \to X$ be a smooth model of $X$ such that $\pi^* \cF/ \tors$ is locally free by the work of \cite{Hir64}, \cite{GR}, \cite{Rie}, \cite{Ros}.
Here $\tors$ means the torsion part of the coherent sheaf.
Without loss of generality, we can assume that it is an isomorphism outside an analytic set of codimension at least 3.
\paragraph{}
\textbf{Main Theorem.} {\em
For a stable reflexive sheaf $\cF$ over a compact, irreducible, smooth in codimension 2 K\"ahler complex space $(X, \omega)$ (e.g. if there exists a boundary divisor $\Delta$ such that $(X, \Delta)$ has terminal singularities), we have the Bogomolov inequality
$$\langle \pi_* ((2r c_2(\pi^* \cF/ \tors)-(r-1)c_1^2(\pi^* \cF/ \tors))\cap [\tilde{X}]) , \omega^{n-2} \rangle \geq 0$$
where $r$ is the generic rank of $\cF$ and $\langle \bullet, \bullet \rangle$ is the pairing between $H_{2n-4}(X,  \R)$ and $H^{2n-4}(X,  \R)$.}
\paragraph{}
The strategy of the proof is as follows. The hypothesis implies that
$\pi^* \cF/ \tors$ is $\pi^*\omega$-stable.
By a result of Toma \cite{Tom16}, stability is an open condition.
In particular, let $\tilde{\omega}$ be a K\"ahler form on $\tilde{X}$. Then
$\pi^* \cF/ \tors$ is $(\pi^* \omega +\varepsilon \tilde{\omega})$-stable for $\varepsilon>0$ small enough.
Thus we have the Bogomolov inequality over $\tilde{X}$ with respect to $\pi^* \omega +\varepsilon \tilde{\omega}$.
The proof is completed by letting $\varepsilon \to 0$.

After the appearance of the current paper on Arxiv, the equality case of our Main Theorem is later studied in \cite{CGG22} following our approach.
\paragraph{}
In general, we expect the Bogomolov inequality to hold on a compact K\"ahler space with klt singularities. But this fact seems much harder to prove (cf. eg. \cite{GKKP11}, \cite{GK20}).
In particular, it is unclear how to define in singular cohomology appropriate ``orbifold" second Chern classes of a coherent sheaf over a compact normal complex space with klt singularities.
\paragraph{}
The organisation of the paper is as follows.
In the second section, for the convenience of readers, we recall some basic results of cohomology theory on (reduced) compact singular spaces.
In particular, we recall the functoriality of the morphism from de Rham cohomology of a complex space to its singular cohomology, a fact that is crucial for our arguments.
In the third section, we compare the ``homology" Chern classes of a coherent sheaf with the Todd class appearing in the Riemann-Roch-Grothendieck theory over singular spaces.
This allows us to prove the Main Theorem in Section 4.

\textbf{Acknowledgement} I thank Jean-Pierre Demailly, my PhD supervisor, for his guidance, patience and generosity. 
I would like to thank Junyan Cao, Patrick Graf, Daniel Greb, Stéphane Guillermou,  Stefan Kebekus, Wenhao Ou, Mihai P\u{a}un, Thomas Peternell and Sheng Rao for some very useful suggestions on the previous draft of this work.
In particular, we thank Professeur Julien Grivaux to point out an error in the previous version.
I would also like to express my gratitude to colleagues of Institut Fourier for all the interesting discussions we had. This work is supported by the European Research Council grant ALKAGE number 670846 managed by J.-P. Demailly.
\section{Cohomology classes on a compact singular space}
Here we recall some basic results of cohomology theory on a (reduced) compact singular space.
Most results of this part are borrowed from Chapiter 7 of the book \cite{AG05}.
To fix the notation, 
we will denote $H^*_{\sim}$ for simplicial cohomology, $\check{H}^*_{}$ for \u{C}ech cohomology, $H^*_{\dR}$ for De Rham cohomology with respect to the complex of smooth forms, and $H^*_{ \sing}$ for singular cohomology defined by continuous cochains.
For $X$ smooth, it is well known that
$$\check{H}^*(X, \R) \simeq H^*_{\dR}(X, \R) \simeq H^*_{ \sing}(X, \R) \simeq H^*_{\sim}(X, \R).$$
For $X$ singular, the canonical map $H^*_{ \sing}(X, \R) \simeq H^*_{\sim}(X, \R)$ is still an isomorphism as a consequence of the triangulation theorem cited below.
To compare the groups $H^*_{ \sing}(X, \R)$, $\check{H}^*(X, \R)$, we also need the triangulation theorem for real analytic varieties.

We recall the following Whitney type embedding theorem for real analytic spaces, due to  Acquistapace, Broglia, Tognoli \cite{ABT79}.
\begin{mythm}
Let $(X, \cO_X^{\R-\an})$ be a reduced real analytic, paracompact, connected space such that\break $\sup_{x \in X} \dim_{\R} T_{x,X} < \infty$.
Here $T_{x,X}$ is the Zariski tangent space at $x$.
Then $(X, \cO_X^{\R-\an})$ can be embedded real analytically in a Euclidean space.
\end{mythm}
Notice that the dimension of the Zariski tangent space, which is locally defined by the rank of  the Jacobian of local generators, is upper semi-continuous.
By the above theorem, we obtain
\begin{mycor}
Let $(X, \cO_X)$ be a compact connected complex space.
Then it can be embedded real analytically in a Euclidean space.
\end{mycor}
On the other hand, we have the following triangulation theorem due to Lojasiewicz for the semianalytic case \cite{Lo64} and to Hardt, Verona for the subanalytic case \cite{Ha74} \cite{Ve79}.
\begin{mythm}
Let $\{ B_\nu \}$ be a locally finite collection of subanalytic subsets of a finite dimensional affine space $M$.
There exist a locally finite simplicial complex $K$ homeomorphic to $M$ and a homeomorphism $\tau : K \to M$
such that:
\begin{enumerate}
\item  $\tau$
is subanalytic i.e. the graph of $\tau$ is subanalytic in $K \times M$.
\item for any $\sigma \in K$ an open simplex and any $B_\nu$, we have $\tau(\sigma) \subset B_\nu$ or $\tau(\sigma) \subset M \setminus B_\nu$.
\item for any $\sigma \in K$ an open simplex, $\tau(\sigma)$ is subanalytic in $M$.
\end{enumerate}
\end{mythm}
By this theorem, a compact complex space has
a triangulation, by taking $X$  as $\{ B_\nu \}$ with a real analytic embedding of $X$ in a Euclidean space.
In particular, $X$ is locally contractible since this is true for a simplicial complex.
Thus we have an isomorphism $H^*_{ \sing}(X, \R) \simeq \check{H}^*(X, \R)$.
In the following, we will identify the simplicial cohomology, the singular cohomology and the \u{C}ech cohomology under the above canonical isomorphisms.

In general, the Poincaré lemma of De Rham cohomology theory does not hold on a singular space; this means the complex of sheaves of smooth forms does not necessarily provide a resolution of the locally constant sheaf $\R$.
In particular, the (smooth) de Rham cohomology may be different from the other cohomology groups.
On the other hand, we have the following resolution by sheaves of subanalytic cochains which gives a flabby resolution of the locally constant sheaf $\R$.
In particular, any element of $\check{H}^*(X, \R)$ can be represented by a global subanalytic cochain.
Since the image of a subanalytic chain under a proper morphism is still subanalytic,
the pull back of a \u{C}ech cohomology class can be represented by the pull back of a global subanalytic cochain.

A K\"ahler form over a (possibly singular) complex space defines a cohomological class in $H^*_{\dR}(X, \R)$.
There exists a natural morphism between the de Rham cohomology and the hypercohomology of the complex of sheaves of subanalytic cochains. It is defined in the following way.

\begin{mydef}(Sheaves of subanalytic (co)chains)

Let $X$ be a real analytic space, $p$ be a non-negative integer. A $p$-subanalytic prechain on $X$  is a pair $(N, c)$, where $N$ is a locally closed subanalytic
subset of $X$ contained in $X$, $\dim N \leq p$, and $c \in H_p (N, \R )$ in the Borel-Moore homology.
Modulo certain equivalent relation, $p$-subanalytic chain on $X$ is defined which forms the $p$-th element in the complex of the subanalytic
chains.
This process can be sheafified to define the complex $(\cS_\bullet,\d)$ of the sheaves of the subanalytic
chains on $X$.
We denote $(\cS^\bullet,\delta)$ for the dual complex of $(\cS_\bullet,\d)$, the complex of the sheaves of the subanalytic
cochains on $X$.
\end{mydef}
The main results that we need in \cite{BH69}, \cite{DP77} can be summarized
in the following  theorem.
\begin{mythm}
Let $X$ be a compact real analytic space. Then
\begin{enumerate}
\item The sheaves $\cS_p$ are soft and the sheaves $\cS^p$ are flabby.
\item $(\cS^\bullet,\delta)$ is a resolution of the locally constant sheaf $\R$.
\item There are natural isomorphisms
$$H_p(X, \R) \simeq H_p (\Gamma(X, \cS_\bullet))$$
and the pairing between the $S^\bullet$ and $S_\bullet$ induces a non-degenerate pairing between $H_p(X, \R)$ and $H^p(X, \R)$.
\end{enumerate}
\end{mythm}
Let $f: X \to Y$ be a morphism between compact real analytic spaces.
If $(N, c)$ is a subanalytic $p$-prechain on
$X$, $f(N)$ is a subanalytic closed subset of $Y$ and the push-forward $f_* c$ is a class in
$H_p (f (N ),\R)$, hence $(f (N ), f_* c)$ is subanalytic $p$-prechain on $Y$; the construction
passes to subanalytic $p$-chains commuting with taking boundaries.
By duality, the pull back of cohomology $f^*: H^p(Y,\R) \to H^p(X, \R)$ is induced by the pull back of global sections of $S^p$.

As in the smooth case, each closed differential form on a complex space defines an element in de Rham cohomology. 
We start by recalling the definition of smooth forms over a (possible singular) complex space
(see e.g. \cite{Dem85}).
\begin{mydef}
Since the definition of smooth forms is local in nature, without loss of generality, we can assume that $X$ is contained in an open set $\Omega \subset \C^N$.
The space of $(p,q)$-forms on $X$ (or $d$-forms on $X$) is defined to be the image of the restriction morphism
$$\cC^\infty_{p,q}(\Omega) \to \cC^\infty_{p,q}(X_\reg)~~~(\mathrm{resp.} \;\cC^\infty_{d}(\Omega) \to \cC^\infty_{d}(X_\reg)).$$
\end{mydef}
\begin{myrem}
{\em 
Let $F: X \to Y$ be a morphism of complex spaces.
Then we have a pull back morphism $F^*: \cC^\infty_{p,q}(Y) \to \cC^\infty_{p,q}(X)$.
In fact, it is checked in Lemma 1.3 of \cite{Dem85} that any local embedding $Y \cap \Omega \subset \Omega \subset \C^N$ and $\alpha \in \cC^\infty_{p,q}(\Omega)$ such that $\alpha|_{Y_\reg}=0$, we have $F^* \alpha|_{X_\reg \cap F^{-1}(\Omega)}=0$. 
}
\end{myrem}
Now we show that the natural morphism between the de Rham cohomology and the hypercohomology of the complex of sheaves of subanalytic cochains is functorial with respect to the pull back by a proper morphism.
Recall first the integration of forms on subanalytic chains.
Let $(N, c)$ be a $p$-subanalytic prechain of a compact complex space $X$ (hence real analytic) and $\omega$ a $p$-differential form on $X$. Let $N^*$ be the set of $p$-regular
points of $N$ (which is open and dense in~$N$) , $i : N^* \to X$ be the embedding.
Define
$$\int_{(N,c)} \omega \coloneqq \int_{N^*} i^* \omega$$
which converges by \cite{DP77}.
Moreover the Stokes formula holds for subanalytic chains.

In particular, the integration of forms on subanalytic chains induces a morphism
$H^p_{dR}(X, \R) \to \Hom(H_p (\Gamma(X, \cS_\bullet)), \R)=H^p (\Gamma(X, \cS^\bullet))=H^p_{\sing}(X,\R). $
We have that
\begin{myprop}
Let $f: X \to Y$ be a morphism between compact complex spaces.
The following diagram commutes
\[
\begin{tikzcd}
H^p_{dR}(Y, \R) \arrow{r}{} \arrow{d}{f^*} & H^p_{\sing}(Y,\R) \arrow{d}{f^*} \\%
H^p_{dR}(X, \R) \arrow{r}{}& H^p_{\sing}(X,\R)
\end{tikzcd}
\]
\end{myprop}
\begin{proof}
If follows from the fact that for any closed form $\alpha$ on $Y$, for any $p$-subanalytic prechain $(N,c)$ on $Y$,
$\int_{f_*(N,c)} \alpha=\int_{(N,c)}f^* \alpha$. 
\end{proof}
\begin{myrem}
{\em
By \cite{Her67}, the morphism $H^p_{dR}(Y, \R)  \to H^p_{\sing}(Y,\R) $ for compact real analytic set $Y$ is always surjective and he gives an example where it is not an isomorphism.
For the simplicity of the reader, we recall briefly the elegant arguments for the surjectivity.
By Theorem 4, the cohomology sheaves of the complex $\cS^\bullet$ satisfy that $\cH^q(\cS^\bullet)=0$ for $q \geq 1$ and that $\cH^0(\cS^\bullet)=\R$.
On the other hand, $\cH^0(\cC^\infty_\bullet)=\R$.
Thus for any $p,q$, $H^p(X, \cH^q(\cS^\bullet)) \to H^p(X, \cH^q(\cC^\infty_\bullet))$ is surjective.
Taking the limit of spectral sequence proves the surjectivity.

Notice that if $X$ is smooth, the natural inclusion of smooth forms into currents induces an isomorphism between the de Rham cohomology defined by forms and the de Rham cohomology defined by currents.
However, if $X$ is singular, this is not always true.
The following example is borrowed from \cite{BH69}.
Let $X$ be the singular curve defined by the compactification of the image of $f:\C \to \C^2 $ sending $t$ to $(z_1,z_2)=(t^5, t^6+t^7)$ (homeomorphic to $\P^1$).
The regular set is $X \setminus \{(0,0)\}$.
Thus $X$ is locally irreducible.
However, $\cH^1(\cC^\infty_\bullet) \neq 0$.
For example, $z_1 dz_2$ (and its complex conjugate) is not exact as germ at $(0,0)$.
Otherwise, locally, $z_1dz_2=d\varphi$ for some germ of smooth function $\varphi$ at $(0,0)$.
$f^* (z_1dz_2)$ is closed near $0$ and thus $f^*(z_1dz_2)=d \psi$ for some germ of smooth function $\psi$ at $0$.
Since $d(f^* \varphi-\psi)=0$ near $0$, $f^*(\varphi)-\psi=\varphi(t^5, t^6+t^7)-\frac{6t^{11}}{11}-\frac{7t^{12}}{12}$ is a constant which is impossible by considering the Taylor development at $0$.
The spectral sequence calculation then shows that the natural map
$$H^1_{\dR}(X, \R) \ncong H^1_{\sing}(X, \R).$$
The following diagram commutes
\[
\begin{tikzcd}
{H^1_{\cD-dR}(X, \R)} \arrow[r] \arrow[d] & {H^1_{\sing}(X, \R)} \arrow[ld] \\
{H^{1}_{\cD'-dR}(X, \R),}                  &                                
\end{tikzcd}
\]
where $H^1_{\cD-dR}(X, \R)$ is the de Rham cohomology defined by forms and $H^1_{\cD'-dR}(X, \R)$ is the de Rham cohomology defined by currents. 
Since the arrow on the first line is surjective but is not an isomorphism, $H^1_{\cD-dR}(X, \R) \to H^1_{\cD'-dR}(X, \R)$ is not an isomorphism.

The factorisation of the morphism $H^1_{\cD-dR}(X, \R) \to H^{1}_{\cD'-dR}(X, \R)$ is constructed as follows.
Let $\pi: \tilde{X} \to X$ be the normalisation of $X$ which is also a desingularisation of $X$ for dimension reason.
We have the following factorisation
\[
\begin{tikzcd}
{H^1_{\cD-dR}(X, \R)} \arrow[d] \arrow[r] & {H^1_{\cD'-dR}(X, \R)}    \\
{H^1_{\cD-dR}(\tilde{X}, \R)} \arrow[r]           & {H^1_{\cD'-dR}(\tilde{X}, \R)}                    \arrow[u] &                      
\end{tikzcd}
\]
Here the vertical arrows are the pull-backs.
The horizontal arrows are induced by the inclusion of forms into currents.
Notice the bottom arrow is isomorphic to $H^1_{\sing}(\tilde{X}, \R) \simeq H^1_{\cD-dR}(\tilde{X}, \R) \simeq H^1_{\cD'-dR}(\tilde{X}, \R)$
since $\tilde{X}$ is a smooth manifold.
Since $X$ is locally irreducible, the normalisation $\tilde{X}$ is homeomorphic to $X$ which implies that
 $H^1_{\sing}(\tilde{X}, \R) \simeq H^1_{\sing}(X, \R)$.
}

\end{myrem}

Note that the space of smooth forms over a compact complex space is a Fr\'echet space.
Since the space is endowed with the quotient topology, it is enough to show that the kernel of $i^*:\cC^\infty_{p,q}(\Omega) \to \cC^\infty_{p,q}(X_\reg)$ is closed where $i$ is the local embedding.
On the other hand, for $\alpha \in \cC^\infty_{p,q}(\Omega)$, $i^* \alpha=0$ if and only if $\alpha \wedge [i(X)]=0$ as a current which is a closed condition.
Then by Lemma 1 \cite{Ser55}, we have isomorphisms
for any compact complex space $X$ of pure dimension $n$, for any $p$,
$$H^p_{\cD-dR}(X, \C)^* \simeq H^{2n-p}_{\cD'-dR}(X, \R)$$
where $H^p_{\cD-dR}(X, \C)^*$ is the topological dual space.

\begin{myrem}
{\em 
Since we have natural factorisation $H^1_{\cD-dR}(X, \C) \to H^1_{\cD-dR}(X, \C)^* \to H^1_{\cD'-dR}(X, \C)$, the non-isomorphism in the above example is equivalent to the failure of Poincaré duality for this example.
The natural Poincar\'e pair has factorisation
$H^1_{\cD-dR}(X, \C) \times H^1_{\cD-dR}(X, \C) \to H^1_{}(X, \C) \times H^1_{}(X, \C) \to \C$.
Since $X$ is homeomorphic to a complex manifold, the second morphism is non degenerated.
The non-surjectivity of $H^1_{\cD-dR}(X, \C) \to H^1_{}(X, \C)$ implies that
the Poincar\'e pair
$H^1_{\cD-dR}(X, \C) \times H^1_{\cD-dR}(X, \C) \to \C$ is however degenerated.
}
\end{myrem}

In the sequel, a complex space $X$ is called Kähler if there exists a smooth $(1,1)$-form on $X$ that is locally the restriction of a Kähler form with respect to a local embedding in an open subset of~$\C^N$.

\section{Homology Chern classes on a compact singular space}
In this section, we give the construction of the homology Chern class (and Chern polynomials) on a compact complex space from the theory of Riemann-Roch-Grothendieck formula over a singular space.
The Riemann-Roch-Grothendieck formula over a singular space is proven by Baum, Fulton and MacPherson \cite{BFM75}, \cite{BFM79} in the projective case and by Levy \cite{Lev87} in the complex analytic case.

Let $K_0(X)$ be the Grothendieck ring of coherent sheaves over $X$ and $K^0(X)$ be the Grothendieck ring of holomorphic vector bundles over $X$.
By an example of Voisin \cite{Voi}, even over smooth compact complex space, these groups are  not necessarily isomorphic.
Even if the space is singular, the Leray-Hirsch theorem defines the Chern classes of a vector bundle in singular cohomology.
Another way to define these is to define the Chern classes of a vector bundle $E$ over $X$ as the pull back of elementary symmetric polynomials which generate the cohomology ring of $G(\rank\;E)$ under the natural morphism from $X$ to the infinite Grassmannian of $(\rank\; E)$-dimensional complex vector spaces $G(\rank \;E)$.
However the definition of Chern classes of a coherent sheaf in this case is much more involved.

To catch only the topological information, the Chern character functor over a compact complex space is defined as follows.
Fix a topological closed embedding $X \subset \C^N $ into some euclidean space such that $X$ is locally contractible (e.g. by the embedding theorem in the previous section).
Let $K^0_\top(X)$ be the
Grothendieck group of topological vector bundles on $X$ and $K_0^\top(X)$ be the Grothendieck group of complexes of vector bundles
on $\C^N$ exact off $X$.
Classically, we have a Chern character functor from $K^0_\top(X)$ to $H^{2*}(X ;\R)$ as recalled above.
By the work of \cite{Lev87}, there exists a group
homomorphism $\alpha: K_0(X) \to K_0^\top(X)$.
For a coherent sheaf $\cF$ over $X$, $\alpha(\cF)$  is given by a complex of vector bundles $E^\bullet$.
To define the Chern character functor $K_0^\top(X) \to H_{2*}(X ;\R)$, we consider the class of $\sum_i(-1)^i \ch(E^i)  \in H^{2*}(\C^N ;\R)$ whose restriction to $\C^N \setminus X$ is 0 since the complex is exact off $X$. 
Since $\C^N$ is orientable and $X$ is compact, locally contractible, the isomorphsim
$H^*(\C^N, \C^N \setminus X, \R) \simeq H_*(X, \R)$ defines the Chern character class (by e.g. Proposition 3.46 in \cite{Hat} and universal coefficient theorem).

The reformulation in the theory of Riemann-Roch-Grothendieck formula over a compact complex space is a natural transformation of functors $\tau :K_0(X) \to H_{2*}(X ;\R)$ such that for any compact complex space $X$ the diagram 
\[
\begin{tikzcd}
K^0(X) \otimes K_0(X) \arrow{r}{\otimes} \arrow{d}{(ch, \tau)} & K_0(X) \arrow{d}{\tau} \\%
H^{2*}(X,\R) \otimes H_{2*}(X,\R)  \arrow{r}{\cap}& H_{2*}(X,\R)
\end{tikzcd}
\]
commutes.
Naturality means that for each proper morphism $f:X \to Y$ of complex spaces the diagram
\[
\begin{tikzcd}
K_0(X) \arrow{r}{\tau} \arrow{d}{f_!} & H_{2*}(X,\R) \arrow{d}{f_*} \\%
K_0(Y) \arrow{r}{\tau}& H_{2*}(Y,\R)
\end{tikzcd}
\]
commutes where $f_!(\cF)=\sum_i R^if_* \cF$ for any coherent sheaf $\cF$ on $X$.
We will denote by $\tau(\cO_X)=\sum_i \tau_i(\cO_X)$ the decomposition of $\tau(X)$ in its homogeneous degree components.
In particular, if $X$ is smooth, $$\tau(\cO_X) = \Todd(X) \cap [X],$$ where $[X]$ is the fundamental homology class of $X$. 
To relate the Chern character functor and the Riemann-Roch-Grothendieck formula, we notice that the Riemann-Roch map $\tau: K_0(X) \to H_{2*}(X,\R)$ factorises through $K_0^\top(X)$.

The functor $\tau$ can be seen a twisted version of the Chern character, such that the Riemann-Roch-Grothendieck formula holds.
If $E$ is a holomorphic vector bundle over a compact complex manifold, we have
$$\tau(E)=\tau(\cO_X) \cap \ch(E)=\ch(E) \cup \Todd(X) \cap [X].$$
In lower degrees,
$$\tau_{2n}(E)=\rk(E)[X], \tau_{2n-2}(E)=(\frac{\rk(E)}{2} c_1(X)+c_1(E)) \cap [X],$$
$$ \tau_{2n-4}(E)=(\frac{\rk(E)}{12} (c^2_1(X)+c_2(X))+\frac{1}{2}c_1(E) c_1(X)+\frac{1}{2}(c_1^2(E)-c_2(E))) \cap [X].$$
In particular, we have
$$c_1(E) \cap [X]=\tau_{2n-2}(E)-\tau_{2n-2}(\cO_X).$$
By the above equality, we can define the homology first Chern class.
We need some preparatory results on the Grothendieck group to do this.
For more information on the Grothendieck ring of coherent sheaves over complex manifold, we refer to the paper \cite{Gri} and to Section 65.14 of \cite{St}.

Let $Z$ be a closed analytic set of some compact complex space $X$.
Let $\coh_Z(X)$ be the set of coherent sheaves on $X$ supported in $Z$.
We have the following equivalent definition.
$$\coh_Z(X)=\{\cF \in \coh(X); \exists N >>0, \cI_Z^N \cF=0\}.$$
It is easy to see that the right-handed side set is contained in $\coh_Z(X)$.
Now we prove the converse inclusion.
Let $i: Z \to X$ be the closed immersion of $Z$ into $X$ with a reduced structural sheaf.
In general, we have a natural morphism for any coherent sheaf $\cF$ on $Z$,
$i^* i_* \cF \to \cF$.
Since $Z$ is closed, by definition, we can easily check that the natural morphism $i^{-1} i_* \cF \to \cF$ is always surjective (which is, in fact, an isomorphism).
Here we denote $i^{-1}$, the inverse image functor in the category of sheaves of abelian groups and $i^*$, the inverse image functor in the category of coherent sheaves.
Since $i^* i_* \cF=i^{-1} i_* \cF \otimes_{i^{-1} \cO_X} \cO_Z$, $i^{-1} \cO_X  \to \cO_Z$ is surjective and taking tensor product is right exact,
the map from $i^{-1} i_* \cF$ to $ i^* i_* \cF$ in the following composition of maps
$$i^{-1} i_* \cF=i^{-1} i_* \cF \otimes_{i^{-1} \cO_X}i^{-1} \cO_X \to i^* i_* \cF \to \cF\otimes_{i^{-1} \cO_X} \cO_Z \to \cF$$
is also surjective.
Since the composition of maps is an isomorphism, $i^* i_* \cF\to \cF$ is also an isomorphism.

Let $\cF$ be a coherent sheaf on $X$ supported in $Z$.
Cover $X$ by finite open sets $U_\alpha$ such that $U_\alpha$ is contained in an open set $\Omega_\alpha$ of $\C^{N_\alpha}$ for some $N_\alpha$.
Let $i_\alpha$ be the inclusion of $U_\alpha$.
$i_{\alpha*} \cF|_{U_\alpha}$ is a coherent sheaf on $\Omega_\alpha$.
Since it is supported in $i_\alpha(Z)$ (with possible shrinking $U_\alpha$ a bit,) there exists $N'_\alpha$ large enough such that
$$\cI_{i_\alpha(Z)}^{N'_\alpha} i_{\alpha*} \cF|_{U_\alpha}=0$$
by Hilbert's Nullstellensatz.
Pull back by $i_\alpha$ gives
$$\cI_{Z}^{N'_\alpha} i^*_\alpha i_{\alpha*} \cF|_{U_\alpha}=0 \to \cI_{Z}^{N'_\alpha} \cF|_{U_\alpha}=0.$$
Since $X$ is compact, $N=\max_\alpha N'_\alpha < \infty$ which shows the inverse inclusion.

Define the Grothendieck group $G_Z(X)$ as $\Z[\coh_Z(X)]$ modulo the following equivalent relation.
$[\cS]+[\cQ]=[\cF] \in G_Z(X)$ if there exists an exact sequence 
$$0 \to \cS \to \cF \to \cQ \to 0$$
for the coherent sheaves $\cS, \cF, \cQ$ supported in $Z$.
Let $G(Z)$ be the Grothendieck group of coherent sheaves on $Z$ (supported in $Z$).
Then we have the following lemma.
\begin{mylem}
We have isomorphism
$$i_{Z*}: G(Z) \simeq G_Z(X).$$
\end{mylem}
\begin{proof}
Notice that since $i_{Z}$ is a closed immersion such that $i_{Z*}$ is exact, the morphism $i_{Z*}$ is well defined between Grothendieck groups.
If $\cF \in \coh_Z(X)$, $\cF$ can be given a finite filtration by the formula $F^i\cF :=\cI^i_Z \cF$.
For any $i$, the associated graded piece $F^i \cF /F^{i+1} \cF$ is in $\coh(Z)$.
Thus we can apply the dévissage theorem to the K-theory group (cf. Theorem 4, Section 5 \cite{Qui}).
\end{proof}
Let $Z_i$ be the irreducible components of $Z$.
We have the following lemma.
\begin{mylem}
The natural morphism
$$\oplus i_{Z_i*}:\oplus_i G_{Z_i}(X) \simeq \oplus_i G(Z_i) \to G_Z(X)$$
is surjective.
\end{mylem}
\begin{proof}
By Lemma 1, it is enough to show that the natural morphism $\oplus_i G(Z_i) \to G(Z)$ is surjective.
We proceed by induction on the number $N$ of the components of $Z$.
Let $Z':=\cup_{i <N} Z_i$.
It is enough to show that $G_{Z'}(Z) \oplus G(Z_N) \to G(Z)$ is surjective.
Let $\cF \in G(Z)$.
Let $\cG$ be the kernel of the natural morphism
$$0 \to \cG \to \cF \to i_{Z_N*}i^*_{Z_N} \cF \to 0.$$
In fact, 
since $Z_N$ is closed, it is easy to see for $z \in Z_N$, $\cF_z \to (i_{Z_N*}i^{-1}_{Z_N} \cF)_z$ is surjective.
On the other hand, 
since $i^{-1}_{Z_N} \cO_X \to \cO_{Z_N}$ is surjective and $i_{Z_N*}$ is right exact,
$$i_{Z_N*}i^{-1}_{Z_N} \cF \to i_{Z_N*}i^{*}_{Z_N} \cF =i_{Z_N*}(i^{-1}_{Z_N} \cF \otimes_{i^{-1}_{Z_N} \cO_X} \cO_{Z_N}) $$
is surjective.
Thus for $z \in Z_N$, $\cF_z \to (i_{Z_N*}i^{*}_{Z_N} \cF)_z$ is surjective.
The surjection outside $Z_N$ is trivial since it is a 0 map.
Over $Z_N \setminus \cup_{i \leq N-1} Z_i$, $i_{Z_N}$ is locally isomorphic, thus $\cF \to i_{Z_N*}i^*_{Z_N} \cF$ is isomorphic.
Thus $\cG$ is supported in $Z'$ and $G_{Z'}(Z) \oplus G(Z_N) \to G(Z)$ is surjective.
(More precisely, by Lemma 65.14.1 \cite{St}, using the adjunction theory, $i_{Z_N*}i^{-1}_{Z_N} \cF=\cF \otimes_{\cO_X} i_{Z_N*} \cO_{Z_N}$. However, we do not need this fact here.)
\end{proof}
\begin{mydef}
For any coherent sheaf $\cF$ on a compact complex space $X$, we define the homology first Chern class of $\cF$ as
$$c_1^h(\cF):=\tau_{2n-2}(\cF)-\tau_{2n-2}(\cO_X).$$
\end{mydef}
In particular, with the natural morphism from de Rham cohomology to singular cohomology, the natural pair between singular cohomology and singular homology 
$$\langle c_1^h(\cF), \omega^{n-1} \rangle$$
defines the slope of $\cF$.
Now, using the Riemann-Roch-Grothendieck formula, we give an equivalent definition of Chern classes by using Hironaka's resolution of singularities.
\begin{mylem}
Let $\cF$ be a coherent sheaf over a compact complex space $X$.
Then for $r$ strictly larger than the dimension of the support of $\cF$,
$$\tau_r(\cF)=0.$$
\end{mylem}
\begin{proof}
Let $A$ be the support of $\cF$ and $i: A \to X$ be the corresponding closed embedding.
Then by Lemma 2, we can take $\cF=i_* \cF'$ and reduce the proof to the case of a coherent sheaf on $A$.
We have $i_! \cF'=i_* \cF'=\cF$, which implies by the Riemann-Roch-Grothendieck formula that $\tau(\cF)=i_* \tau(\cF')$.
Notice that by a triangulation of $A$, we have $H_i(A, \R)=0$ for $i$ strictly larger than the dimension of $A$.
This finishes the proof.
\end{proof}
\begin{mylem}
Let $X$ be a reduced, compact, irreducible, complex space that is smooth in codimension 1. Let $\cF$ be a torsion-free coherent sheaf over $X$.
Let $\pi: \tilde{X} \to X$ be a resolution of singularity such that $\tilde{X}$ is smooth and $\pi^* \cF/ \tors$ is a vector bundle where $\tors$ means the torsion part.
Without loss of generality, we can assume that $\pi$ is a biholomorphism on an open set of $X$ outside a closed analytic set of codimension at least 2.
Then we have 
$$c_1^h(\cF)=\pi_*(c_1(\pi^* \cF/ \tors) \cap [\tilde{X}])$$
If $X$ is smooth, we also have that
$$c_1^h(\cF)=c_1(\det(\cF)) \cap [X].$$
\end{mylem}
\begin{proof}
By the assumption that $\pi^*(\cF)/ \tors$ is a vector bundle, the Riemann-Roch-Grothendieck theorem implies that
$$\tau(\pi^*(\cF)/ \tors)=\ch(\pi^*(\cF)/ \tors)) \cup \Todd(\tilde{X}) \cap [\tilde{X}].$$
Observe that $\pi_* (\pi^*(\cF)/ \tors)$, $\cF$ coincide outside an analytic set of (complex) codimension at least~$2$.
By the codimension assumption, for any $i>0$, the support of $R^i \pi_* (\pi^*(\cF)/ \tors)$ is of (complex) codimension at least 2.
This implies
$$\tau_{2n-2}(\cF)=\pi_* \tau_{2n-2}(\pi^*(\cF)/ \tors)$$
by Lemma 3. 
Thus we have $$\tau_{2n-2}(\cF)-\tau_{2n-2}(\cO_X)=\pi_* (\tau_{2n-2}(\pi^*(\cF)/ \tors)-\tau_{2n-2}(\cO_{\tilde{X}}))=\pi_*(c_1(\pi^*(\cF)/ \tors) \cap [\tilde{X}]).$$

Now we begin to prove the second equality. 
By duality between $H^2(X,\R)$ and $H_2(X, \R)$, the second equality is equivalent to the fact that for any $\alpha \in H^2(X, \R)$
$$\langle \pi_*(c_1(\pi^* \cF/ \tors) \cap [\tilde{X}]), \alpha \rangle=\langle c_1(\det(\cF)) \cap [X], \alpha \rangle.$$
The left hand side term is equal to 
$$\langle (c_1(\pi^* \cF/ \tors) \cap [\tilde{X}]), \pi^* \alpha \rangle=\langle (c_1(\pi^* \cF/ \tors) \cup \pi^* \alpha), [\tilde{X}]  \rangle.$$
The difference $c_1(\pi^* \cF/ \tors)-\pi^*(c_1( \det \cF))$ is a linear combination of the cohomology classes $\{E_i\}$ associated to the irreducible components of the exceptional divisor where $i_{E_i,*}$ is the closed embedding of $E_i$.
By Poincaré duality (by $\tilde{X}$ compact smooth), $\langle \{E_i\} \cup \pi^* \alpha, [\tilde{X}] \rangle=\langle i_{E_i}^*  \pi^* \alpha, [E_i]\rangle$.
On the other hand, in singular cohomology,$$i_{E_i}^*  \pi^* \alpha=\pi_{|_{E_i}}^* i_{\pi(E_i)}^* \alpha=0$$
since $i_{\pi(E_i)}^* \alpha=0$ by the dimension condition of $\pi(E_i)$.
Hence we have 
$$\langle \pi_*(c_1(\pi^* \cF/ \tors) \cap [\tilde{X}]), \alpha \rangle=\langle \pi^*(c_1( \det \cF) \cup  \alpha), [\tilde{X}]  \rangle=\langle (c_1( \det \cF) \cup  \alpha), \pi_*[\tilde{X}]  \rangle$$
which finishes the proof since $\pi_*[\tilde{X}]=[X]$.
\end{proof}
\begin{myrem}
{\em 
If $\cF$ admits a resolution of vector bundles (e.g. if $X$ is a projective manifold), we have for a resolution $E^\bullet$ that
$$\tau(\cF)=\sum_i \tau(E_i)=\sum_i (-1)^i \ch(E_i) \tau(\cO_X).$$
In this case, one way to define the Chern character class in cohomology is to define
$$\ch(\cF)= \sum_i (-1)^i \ch(E_i).$$
We have
$$\tau_{2n-2}(\cF)-\tau_{2n-2}(\cO_X)=(\sum_i (-1)^i c_1(E_i) ) \cap [X]=c_1(\det(\cF)) \cap [X].$$
In other words, $c_1^h(\cF)$ is the virtual ``Poincaré dual" of $c_1(\cF)$.
}
\end{myrem}
More generally, if $X$ is smooth in codimension $k$ and $\cF$ is locally free in codimension $k$, one can define the homology Chern polynomial of degree $\leq k$ as
$$P(c_\bullet(\cF))^h:=\pi_*(P(c_\bullet (\pi^* \cF/ \tors)) \cap [\tilde{X}])$$
for some modification $\pi: \tilde{X} \to X$ such that $\tilde{X}$ is smooth and $\pi^* \cF/ \tors$ is locally free.
Without loss of generality, we can assume that $\pi$ is a biholomorphism over an open set of $X$ outside closed analytic set of codimension at least $k+1$.
By filtering property of modifications, this is independent of the choice of such modification $\pi$.
If $\cF$ is locally free, 
$$P(c_\bullet(\cF))^h=P(c_\bullet ( \cF)) \cap [X].$$
\begin{mylem}
For any exact sequence 
$$0 \to \cS \to \cF \to \cQ \to 0$$
where $X$ is smooth in codimension $k$ and $\cS,\cF, \cQ$ are locally free in codimension $k$,
we have
$$\ch_{\leq k}^h(\cF)=\ch_{\leq k}^h(\cS)+\ch_{\leq k}^h(\cQ).$$
\end{mylem}
\begin{proof}
We have a long exact sequence 
$$\cdots \to L^i\pi^* \cS \to L^i \pi^* \cF \to L^i \pi^*  \cQ \to L^{i-1} \pi^* \cS \to \cdots .$$
Notice that for $i>0$, $L^i \pi^* \cF $ is a torsion sheaf supported in the exceptional divisor as well as the torsion part of $\pi^* \cF$.

To prove the lemma, it is enough to show that
for any torsion sheaf $\cF$ supported in the exceptional divisor $E$ on $\tilde{X}$ and any polynomial $P$ of degree $\leq k$ in terms of $c_\bullet(\cF)$ and $\beta$ any cohomology class on $\tilde{X}$, 
$$\pi_*(P(c_\bullet(\cF), \beta) \cap [\tilde{X}])=0.$$
By the duality between singular cohomology and singular homology,
it is enough to show that for any singular cohomology class $\alpha$ on $X$
 $$\langle \pi_*(P(c_\bullet(\cF), \beta) \cap [\tilde{X}]), \alpha \rangle=\langle P(c_\bullet(\cF), \beta) \cap [\tilde{X}], \pi^* \alpha \rangle =0.$$
Apply Proposition 1 to $\pi$.
Since $\tilde{X}$ is smooth, the bottom arrow in Proposition 1 is in fact an isomorphism.
Notice also that the de Rham cohomology of $X$ maps surjectively onto the singular cohomology of $X$.
Thus without loss of generality, we can assume that $\alpha$ can be defined in de Rham cohomology.
By the construction of Chern class over smooth manifolds (cf. \cite{Gri}, \cite{Wu19}), 
$c_\bullet(\cF)=i_{E*}(\gamma_\bullet)$ for some cohomology classes $\gamma_\bullet$.
Thus we have
$$\langle P(c_\bullet(\cF), \beta) \cap [\tilde{X}], \pi^* \alpha \rangle=\int_{\tilde{X}} P(c_\bullet(\cF), \beta) \wedge \pi^* \alpha=\int_E P(\gamma_\bullet, i_E^* \beta) \wedge i_E^* \pi^* \alpha=0,$$
since $i_E^* \pi^* \alpha=0$ by a dimension argument.
\end{proof}
In particular, if $\cF$ is a reflexive sheaf over a compact irreducible smooth in codimension 2 space and if $\cF$ admits a resolution of vector bundles, we have for a resolution $E_\bullet$ that
$$\ch_{\leq 2}^h(\cF)=\sum_i (-1)^i \ch_{\leq 2}^h(E_i)=\sum_i (-1)^i \ch_{\leq 2}(E_i)\cap [X].$$
\begin{myrem}
{\em
One can also define Chern polynomials in cohomology.
If $X$ is smooth in codimension~$k$ and $\cF$ is locally free in codimension $k$, one can define the Chern polynomial of degree $\leq k$ as
$$P(c_\bullet(\cF)):=\pi_*(P(c_\bullet (\pi^* \cF/ \tors))) \in \oplus_{i \leq k} H^{2i}_{\cD'-dR}(X, \R)$$
for some modification $\pi: \tilde{X} \to X$ such that $\tilde{X}$ is smooth and $\pi^* \cF/ \tors$ is locally free.
Without loss of generality, we can assume that $\pi$ is a biholomorphism over an open set of $X$ outside a closed analytic set of codimension at least $k+1$.
We have the following commutative diagram:
\[
\begin{tikzcd}
H_{2n-*}(\tilde{X}, \R)\arrow{d}{} \arrow{r}{}  & H^{2n-*}(\tilde{X}, \R)^*  \arrow{r}{} \arrow{d}{}  & H_{\cD-dR}^{2n-*}(\tilde{X}, \R)^* \arrow{r}{} \arrow{d}{} & H^{*}_{\cD'-dR}(\tilde{X}, \R) \arrow{d}{}  \\%
H_{2n-*}(X, \R) \arrow{r}{}  & H^{2n-*}(X, \R)^* \arrow{r}{}  & H^{2n-*}_{\cD-dR}(X, \R)^* \arrow{r}{}  & H^{*}_{\cD'-dR}(X,  \R). 
\end{tikzcd}
\]
The first horizontal morphism is given by duality between singular cohomology and homology.
The last horizontal morphism is induced by the natural pair between currents and forms.
Since $\tilde{X}$ is smooth, all arrows on the first line are isomorphisms. Then
$P(c_\bullet(\cF))$ is exactly the image of $P(c_\bullet(\cF))^h$ in $\oplus_{i \leq k} H_{2n-2i}(X, \R)$.
Since the morphisms on the bottom arrows are not necessarily isomorphisms if $X$ is singular, in some sense, the homology version of a Chern polynomial is more precise than the corresponding definition in cohomology.

The dimension condition in the definition of a homology Chern polynomial is also reflected in the following diagram.
Let $P(c_\bullet(\cF))$ be a Chern polynomial of degree $i \leq k$.
Let $Z$ be a closed analytic set of codimension at least $k+1$ such that $X\setminus Z$ is smooth and $\cF|_{X\setminus Z}$ is locally free.
Assume that $\pi$ is chosen such that it is a biholomorphism over $X\setminus Z$ with exceptional divisor $E \subset \tilde{X}$. 
On the other hand, we have the following exact sequence in singular cohomology
$$H^{2n-2i-1}(Z, \R) \to H^{2n-2i}_c(X \setminus Z, \R) \to H^{2n-2i}(X, \R) \to H^{2n-2i}(Z, \R)$$
for which $H^{2n-2i-1}(Z, \R)=H^{2n-2i}(Z, \R)=0$ by the codimension condition.
We have also that $H^{2n-2i}_c(X \setminus Z, \R)^* \simeq H^{2i}(X \setminus Z, \R)$ as a consequence of the Poincaré duality.
We have the following commutative diagram
\[
\begin{tikzcd}
H_{2n-2i}(\tilde{X}, \R)\arrow{d}{} \arrow{r}{\simeq}  & H^{2n-2i}(\tilde{X}, \R)^*  \arrow{r}{} \arrow{d}{}  & H_{c}^{2n-2i}(\tilde{X} \setminus E, \R)^* \arrow{r}{\simeq} \arrow{d}{\simeq} & H^{2i}_{}(\tilde{X} \setminus E, \R) \arrow{d}{\simeq}  \\%
H_{2n-2i}(X, \R) \arrow{r}{\simeq}  & H^{2n-2i}(X, \R)^* \arrow{r}{\simeq}  & H^{2n-2i}_{c}(X \setminus Z, \R)^* \arrow{r}{\simeq}  & H^{2i}_{}(X \setminus Z,  \R). 
\end{tikzcd}
\]
The image of $P(c_\bullet(\pi^* \cF/ \tors)) \cap [\tilde{X}]$ in $H^{2i}_{}(\tilde{X} \setminus E, \R) \simeq H^{2i}_{}(X \setminus Z,  \R)$
is $P(c_\bullet(\cF|_{X \setminus Z}))$.
In particular, with this definition, the definition of homology Chern polynomial is independent of $\tilde X$ and can be defined without resolution of singularities.
}
\end{myrem}
The coincidence with the algebraic case is shown in the following proposition.
\begin{myprop}
Let $X$ be a irreducible projective variety.
Let $\cF$ be a torsion-free coherent sheaf over $X$.
Assume that $X$ is smooth in codimension $k$ and $\cF$ is locally free in codimension $k$.
Let $P$ be a Chern polynomial of degree $\leq k$.
Let $A$ be a very ample ample divisor on $X$ and denote $\omega_A$ the K\"ahler form in the class of $A$.
Let $S$ be a smooth variety as complete intersection of $H_1, \cdots, H_{n-k} \in |A|$ such that $\cF$ is locally free near $S$.
(Notice that such $S$ always exists if $A$ is taken to be sufficiently ample.)
Then we have
$$\langle P(c_\bullet(\cF))^h, \omega^{2n-2k}_A \rangle=\int_S P(c_\bullet(\cF|_{S})).$$
\end{myprop}
\begin{proof}
We have
$$\langle P(c_\bullet(\cF))^h, \omega^{2n-2k}_A \rangle=\int_{\tilde{X}} P(c_\bullet(\pi^* \cF/ \tors)) \wedge \pi^* \omega_A^{2n-2k}.$$
Without loss of generality, we can assume that $\pi$ is a biholomorphism near $S$.
In particular, $\pi^* \omega_A$ is cohomologous to $\pi^* H_i$ and $\pi^* \omega_A^{n-k}$ is cohomologous to $\pi^* H_1 \cap \cdots \cap \pi^* H^{n-k} \simeq S$ by \cite{Dem93}.
Thus the left hand side is equal to
$\int_S P(c_\bullet(\cF|_{S}))$ since $\pi^* \cF/ \tors|_S=\cF|_S$.
\end{proof}
\section{The Bogomolov inequality}
In this section, we show that the Bogomolov inequality holds over a compact, irreducible, K\"ahler space which is smooth in codimension 2.
The disadvantage of the topological proof is in substance that it does not give further information on the case of equality.
The existence of an admissible metric as defined in Bando-Siu does not seem to be a trivial question over a singular space.
Over a singular space we obtained the following version of the Bogomolov inequality.
\begin{mythm}
For a stable reflexive sheaf $\cF$ over a compact, irreducible, K\"ahler complex space $(X, \omega)$ that is smooth in codimension $2$ (e.g. if there exists a boundary divisor $\Delta$ such that $(X, \Delta)$ has terminal singularities), we have the Bogomolov inequality
$$\langle (2r c_2^h(\cF)-(r-1)c_1^{2,h}(\cF)) , \omega^{n-2} \rangle \geq 0$$
where $r$ is the generic rank of $\cF$ and $\langle \bullet, \bullet \rangle$ is the pair between $H_{2n-4}(X,  \R)$ and $H^{2n-4}(X,  \R)$.
\end{mythm}
\begin{proof}
By definition, we have
$$ 2r c_2^h(\cF)-(r-1)c_1^{2,h}(\cF)=\pi_*((2r c_2(\pi^* \cF/\tors)-(r-1)c_1^2(\pi^* \cF/\tors) )\cap [\tilde{X}])$$
for some modification $\pi: \tilde{X} \to X$ such that $\tilde{X}$ is smooth and $\pi^* \cF/\tors$ is locally free.
Then by the functoriality given by Proposition 1
$$\langle (2r c_2^h(\cF)-(r-1)c_1^{2,h}(\cF)) , \omega^{n-2} \rangle = \int_{\tilde{X}} (2r c_2(\pi^* \cF/\tors)-(r-1)c_1^2(\pi^* \cF/\tors) \wedge \pi^* \omega^{n-2}.$$
By Lemma 4, $\cF$ is $\omega$-stable implies that $\pi^* \cF/\tors$ is $\pi^* \omega$-stable.
In fact for any proper torsion-free subsheaf $\cF'$ of $\pi^* \cF/\tors$,
$\pi_* \cF'$ is a subsheaf of $\pi_*(\pi^* \cF/\tors)=\cF$.
Since $\pi$ is a modification which preserves the generic rank,
$\pi_* \cF'$ is a proper subsheaf.
By the stability condition,
$$\langle c_1^h(\pi_* \cF'), \omega^{n-1} \rangle < \langle c_1^h( \cF), \omega^{n-1} \rangle. $$
On the other hand,
$c_1^h(\pi_* \cF')=\pi_*( c_1(\cF') \cap [\tilde{X}])$, $c_1^h( \cF)=\pi_*( c_1(\pi^* \cF/ \tors ) \cap [\tilde{X}])$.
Thus the above inequality is equivalent to
$$\int_{\tilde{X}} c_1( \cF') \wedge \pi^* \omega^{n-1}  < \int_{\tilde{X}} c_1( \pi^* \cF/ \tors) \wedge \pi^* \omega^{n-1}. $$
Since stability is an open condition, for $\delta>0$ small enough, $\pi^* \cF/ \tors$ is $(\pi^* \omega +\delta \tilde{\omega})$-stable where $\tilde{\omega}$ is a K\"ahler form on $\tilde{X}$ by Corollary 6.9 of \cite{Tom19}.
We sketch the proof for the convenience of the readers.
In fact, by Corollary 6.3 of \cite{Tom19}, there is a finite number of irreducible components of the Douady space containing destabilizing quotients of the vector bundle $\pi^* \cF/ \tors$ with respect to $\tilde{\omega}$, i.e. proper quotients $\cQ$ with $\mu_{\tilde{\omega}} (\cQ) \leq \mu_{\tilde{\omega}} (\pi^* \cF/ \tors)$.
Let $\cQ_1, \cdots, \cQ_k$ be the proper quotients of $\pi^* \cF/ \tors$ representing each irreducible component with
$$\mu_{\tilde{\omega}} (\cQ_i) \leq \mu_{\tilde{\omega}} (\pi^* \cF/ \tors)~~(\forall i,~~1 \leq i \leq k).$$
Since for any $1 \leq i \leq k$,
$$\delta \mapsto \mu_{
\pi^*{\omega}+\delta\tilde{\omega}} (\cQ_i)-\mu_{\pi^*{\omega}+\delta\tilde{\omega}} (\pi^* \cF/ \tors)$$
is affine and strictly positive at $\delta=0$,
it is strictly positive near $\delta=0$.
For any proper quotient $\cQ$ of $\pi^* \cF/ \tors$, either $\cQ$ is not destabilizing, in which case
$$\mu_{\pi^*{\omega}+\delta\tilde{\omega}} (\cQ) > \mu_{\pi^*{\omega}+\delta \tilde{\omega}} (\pi^* \cF/ \tors),$$
or there exists some $i$ such that $\cQ, \cQ_i$ are in the same irreducible component of the Douady space, in which case, for $\delta>0$ small enough,
$$\mu_{
\pi^*{\omega}+\delta\tilde{\omega}} (\cQ)=\mu_{
\pi^*{\omega}+\delta\tilde{\omega}} (\cQ_i)>\mu_{\pi^*{\omega}+\delta\tilde{\omega}} (\pi^* \cF/ \tors).$$
By Theorem 1 cited above, we have
$$\int_{\tilde{X}} (2r c_2(\pi^* \cF/ \tors)-(r-1)c_1^2(\pi^* \cF/ \tors)) \wedge (\pi^*\omega+\delta \tilde{\omega})^{n-2} \geq 0.$$
By letting $\delta$ tend to 0, we find
$$\int_{\tilde{X}} (2r c_2(\pi^* \cF/ \tors)-(r-1)c_1^2(\pi^* \cF/ \tors)) \wedge (\pi^*\omega)^{n-2} \geq 0.$$
This finishes the proof.
\end{proof}
\begin{myrem}
{\em 
The same proof works if we change $\omega$ into any closed semi-positive form on the K\"ahler complex space.
}
\end{myrem}
\begin{myrem}
{\em 
In a previous version of this paper, we had tried to use a resolution of a coherent sheaf by real analytic vector bundles of finite length to define Chern classes in singular cohomology.
As pointed to us by Grivaux, once the complex space is not smooth, the homological dimension is infinite. In particular, the expected resolution described above does not necessarily exist.

In particular, to define the second ``orbifold" Chern class over a compact normal complex space with klt singularities, we have to define Chern classes of a coherent sheaf over a compact orbifold.
In this case, singular (co)homology does not seem to be the appropriate (co)homology theory to work with. 
}
\end{myrem}
We also have the following version of the Fulton-Lazarsfeld inequalities.
We start by giving the definition of a nef torsion free sheaf.
\begin{mydef}
Assume that $\cF$ is a torsion free sheaf over an irreducible compact complex space $X$.
We say that $\cF$ is nef if there exists some modification
$\pi: \tilde{X} \to X$ such that $\pi^* \cF/\tors$ is a nef
 vector bundle over $\tilde{X}$ a smooth manifold, where $\tors$ means the torsion part.
\end{mydef}
\begin{myprop}
Let $(X,\omega)$ be a compact, irreducible K\"ahler space smooth in codimension $k$ of dimension $n$.
Let $\cF$ be a nef torsion free sheaf over $X$ which is locally free in codimension $k$.
Let $r$ be the generic rank of $\cF$ and assume $r \leq k$.
Then we have
$$ \langle c^h_r(\cF) , \omega^{n-r} \rangle \geq 0.$$
More generally, for any Schur polynomial $P_a$ of degree $ \leq 2r$,
$$ \langle P^h_a(c_\bullet(\cF)) , \omega^{n-r} \rangle \geq 0.$$
\end{myprop}
\begin{proof}
The proof is analogous to the proof of the Bogomolov inequality by changing Theorem 1 into Proposition 2.1 and Theorem 2.5 of \cite{DPS94}.
Notice that the integration of higher degree Chern classes is automatically 0 by a degree argument.
\end{proof}

Key words: Bogomolov inequality, coherent sheaf, Chern class.

Universität Bayreuth,
Universitätsstraße 30, 95447 Bayreuth

E-mail address: Xiaojun.Wu@uni-bayreuth.de
\end{document}